\documentclass{amsart}

\usepackage[margin=1.4in]{geometry}

\usepackage{graphicx} 
\usepackage{enumerate}  
\usepackage{hyperref}   
\usepackage[square,numbers]{natbib}     
\usepackage{amsthm}     
\usepackage{amsmath}    
\usepackage{amssymb}    
\usepackage{amsxtra}    
\usepackage{eucal}      
\usepackage{comment}
\usepackage{tikz}

\usetikzlibrary{matrix}
\usetikzlibrary{patterns, shapes}
\usetikzlibrary{arrows}
\usetikzlibrary{calc,3d}
\usetikzlibrary{decorations,decorations.pathmorphing, decorations.pathreplacing}
\usetikzlibrary{through}
\tikzset{every picture/.style={baseline=-.65ex}}
\tikzset{ext/.style={circle, draw,inner sep=1pt},int/.style={circle,draw,fill,inner sep=1pt},nil/.style={inner sep=1pt}}
\tikzset{exte/.style={circle, draw,inner sep=3pt},inte/.style={circle,draw,fill,inner sep=3pt}}
\tikzset{diagram/.style={matrix of math nodes, row sep=3em, column sep=2.5em, text height=1.5ex, text depth=0.25ex}}
\tikzset{diagram2/.style={matrix of math nodes, row sep=0.5em, column sep=0.5em, text height=1.5ex, text depth=0.25ex}}
\tikzset{every loop/.style={draw}}

\usepackage{tikz-cd}
\usepackage{ytableau} 

\theoremstyle{plain}

\newtheorem{thm}{Theorem}[section]

\newtheorem{prop}[thm]{Proposition}

\newtheorem{lemma}[thm]{Lemma}
\newtheorem{cor}[thm]{Corollary}

\theoremstyle{definition}
\newtheorem{rem}[thm]{Remark}

\newtheorem{defi}[thm]{Definition}
\newtheorem{ex}[thm]{Example}

\newcommand{\Q}{\mathbb{Q}}

\newcommand{\GC}{\mathsf{GC}}
\newcommand{\HGC}{\mathsf{HGC}}

\DeclareMathOperator{\Hom}{Hom}
\DeclareMathOperator{\Map}{Map}
\newcommand{\lp}{\text{-loop}}

\DeclareMathOperator{\ad}{ad}
\newcommand{\Com}{\mathrm{Com}}

\newcommand{\Z}{\mathbb{Z}}

\newcommand{\grt}{\mathfrak{grt}}
\newcommand{\ogrt}{\mathfrak{ogrt}}
\usepackage{todonotes}

\newcommand{\Lie}{\mathrm{Lie}}

\newcommand{\ot}{{\mathfrak{ot}}}

\newcommand{\Der}{\operatorname{Der}}

\theoremstyle{remark}

\newcommand{\Pois}{\mathsf{Pois}}

\title[An odd version of the Grothendieck--Teichmüller Lie algebra]{An odd version of the Grothendieck--Teichmüller Lie algebra and the odd Deligne--Drinfeld conjecture}
\author{Florian Naef}
\address{School of Mathematics, Trinity College, Dublin 2, Ireland }
\author{Thomas Willwacher }
\address{Department of Mathematics, ETH Zurich, Rämistrasse 101, 8092 Zurich, Switzerland}
\date{}

\begin{document}

\begin{abstract}
    We study an ``odd'' version of the Grothendieck--Teichmüller Lie algebra that is obtained by replacing free Lie algebras in even generators by free graded Lie algebras in odd generators in the classical definition.
    These odd variants are related to the ``odd'' Kontsevich graph complex in the same way as the classical Grothendieck--Teichmüller group is related to the even Kontsevich graph complex. We show an odd analog of the Deligne--Drinfeld conjecture, namely that the odd Grothendieck--Teichmüller Lie algebra is isomorphic to a free graded Lie algebra in one odd generator of weight 3. Thus we in particular compute the odd graph cohomology in the bottom cohomological degree, showing that it is mostly zero.
\end{abstract}

\maketitle

\section{Introduction}
The Grothendieck--Teichmüller Lie algebra $\grt_1:=\grt_1(\Q)$ is an infinite-dimensional Lie algebra introduced by Drinfeld \cite{Drinfeld}.
It has been studied extensively in the last decades. The Deligne--Drinfeld conjecture states that $\grt_1\cong \Lie(\sigma_3,\sigma_5,\dots)$ is identified with a free Lie algebra in infinitely many generators. This conjecture has recently been proven by OpenAI \cite{OpenAI}, extending earlier work by Brown \cite{Brown}, so that $\grt_1$ is today fully known.

The Grothendieck--Teichmüller Lie algebra appears in a multitude of a priori unrelated problems in algebra, topology and mathematical physics. For example, it has been shown by the second author \cite{WillwacherGrt} that $\grt_1$ computes the zeroth cohomology of the Kontsevich graph complex $\GC_2$, $H^0(\GC_2)\cong \grt_1$. Since it is known that $H^{<0}(\GC_2)=0$, see \cite{WillwacherGrt}, this result can be seen as the full computation of the graph cohomology in the bottom degree.
Now there exist two variants of the Kontsevich graph complex, the even variant $\GC_2$ and an odd variant $\GC_1$, that also has a multitude of applications. For example, the homotopy groups of the homotopy automorphism groups of the rational $E_n$-operad may be fully expressed through $H(\GC_2)$ if $n\geq 2$ is even and through $H(\GC_1)$ if $n\geq 3$ is odd, see \cite[Corollary~5]{FTW17}. So there is a significant interest in computing $H(\GC_1)$ as well. 
In particular, one may ask whether the techniques that enabled computing the bottom degree cohomology of $\GC_2$ can be transferred to the odd setting to compute the bottom degree cohomology of $\GC_1$ as well.
In this paper we give an affirmative answer to this question.

{\bf The odd Grothendieck--Teichmüller Lie algebra and main results}
Let $\ot_n$ be the odd Drinfeld--Kohno Lie algebra. That is, $\ot_n$ is the graded Lie algebra generated by symbols $t_{ij}$, $1\leq i\neq j\leq n$ of cohomological degree 1 and weight 1, with the following relations:
\begin{align*}
t_{ij}+t_{ji}&=0 \\
    [t_{ij}, t_{kl}] &= 0 &\text{for $i,j,k,l$ pairwise distinct} \\
    [t_{ij}, t_{ik}+t_{jk}] &= 0 & \text{for $i,j,k$ pairwise distinct}.
\end{align*}
Introduce two variables $x,y$ of cohomological degree 1 and weight 1,
and let $\Lie(x,y)$ be the free graded Lie algebra in $x$ and $y$. Note that in the free graded Lie algebra the antisymmetry relation of the Lie bracket is 
\[
[u,v] = -(-1)^{|u||v|} [v,u]
\]
on homogeneous elements $u,v\in \Lie(x,y)$ of cohomological degrees $|u|$ and $|v|$, respectively. In particular, $[x,x],[y,y]\neq 0$ and the free Lie algebra in a single odd generator is two-dimensional.
For an odd element $u$, we write $u^2=\frac12[u,u]$, in agreement with its square in the universal enveloping algebra. 

\begin{defi}
    The odd Grothendieck--Teichmüller Lie algebra $\ogrt_1\subset \Lie(x,y)$ consists of $\psi(x,y)\in \Lie(x,y)$ satisfying the following system of homogeneous linear equations 
    \begin{align}
\psi(x,y)+\psi(-y,-x)&=0,\label{eq:oantisymmetry}
\qquad \qquad\text{(antisymmetry)}
\\
\psi(t_{12},t_{23})+\psi(t_{23},t_{31})+\psi(t_{31},t_{12})&=0,\label{eq:othree-cycle}
\qquad \qquad\text{(3-cycle or hexagon)}
\\
\psi(t_{12},t_{23}+t_{24})+\psi(t_{13}+t_{23},t_{34})
 &=\psi(t_{23},t_{34})+\psi(t_{12}+t_{13},t_{24}+t_{34})
 \notag\\
 &\qquad
   +\psi(t_{12},t_{23}), \qquad\text{(pentagon),}
   \label{eq:opentagon}
\end{align}
where \eqref{eq:othree-cycle} takes values in $\ot_3$ and \eqref{eq:opentagon} takes values in $\ot_4$.
\end{defi}

For a homogeneous Lie polynomial $f$ let $\partial_f$ be the derivation
of parity $|f|$ determined by
\begin{equation}\label{eq:ihara-action}
 \partial_f x=0,\qquad \partial_f y=-[f,y].
\end{equation}
We use the graded Ihara bracket
\begin{equation}\label{eq:ihara}
 \{f,g\}=[f,g]+\partial_f g-(-1)^{|f||g|}\partial_g f.
\end{equation}
This bracket endows $\ogrt_1$ with the structure of a graded Lie algebra, as we will see. It is graded by weight or equivalently cohomological degree, and we denote the part of weight (or degree) $k$ by $\ogrt_1^{(k)}\subset \ogrt_1$.

\begin{ex}\label{ex:lowweight}
In weight 1 the only solution of \eqref{eq:oantisymmetry} is $\psi=\alpha (x+y)$ for $\alpha\in \Q$. But then \eqref{eq:othree-cycle} implies $\alpha=0$, so that 
\[
\ogrt_1^{(1)}=0.
\]
In weight 2, \eqref{eq:oantisymmetry} has the general solution $\psi=\alpha([x,x]-[y,y])$. Then \eqref{eq:othree-cycle} is automatically satisfied. Inserting this $\psi$ into \eqref{eq:opentagon} and projecting the resulting equation to $\ot_2$ by forgetting strands $3$ and $4$ we obtain
\[
\alpha [t_{12},t_{12}] + 0 = 0 + \alpha [t_{12},t_{12}] + \alpha [t_{12},t_{12}].
\]
Hence $\alpha=0$, so that
\[
\ogrt_1^{(2)}=0
\]
as well.
The first non-trivial element arises in weight 3. Concretely, one checks that
\begin{equation}\label{eq:sigma}
 \sigma_3=[x+y,[x,y]]
\end{equation}
is an element of $\ogrt_1$ of weight and degree 3. Its self-bracket of degree 6
\begin{equation}\label{eq:s3s3}
\begin{aligned}
\{\sigma_3,\sigma_3\}={}&
\big[[x,[x,y]],[x,[x,y]]\big]-\big[[y,[y,x]],[y,[y,x]]\big]-\frac14\big[[x,x],[[x,x],[y,y]]\big]\\
&-\frac14\big[[y,y],[[x,x],[y,y]]\big]+\frac12\big[[x,y],[[x,x],[y,y]]\big]
\end{aligned}
\end{equation}
is an element of $\ogrt_1^{(6)}$ as well, as one may check by explicit verification of the defining equations \eqref{eq:oantisymmetry}-\eqref{eq:opentagon}. It can be seen to be non-zero by considering the coefficient of the word $xxyxxy$.
\end{ex}

Our main result states that these two elements $\sigma_3,\{\sigma_3,\sigma_3\}$ above already span $\ogrt_1$.

\begin{thm}[Odd analog of the Deligne--Drinfeld conjecture]\label{thm:main}
The map of graded Lie algebras 
\[
 \Lie(\sigma_3) = \Q\sigma_3\oplus\Q\{\sigma_3,\sigma_3\}\longrightarrow \ogrt_1
\]
sending $\sigma_3$ to the degree 3 element \eqref{eq:sigma} is an isomorphism. Thus
\[
 \dim\ogrt_1^{(n)}=
 \begin{cases}1,&n=3,6,\\0,&n\ne3,6.\end{cases}
\]
\end{thm}

The above result has an immediate application to the odd graph cohomology.
Analogously to the main result of \cite{WillwacherGrt} we show:
\begin{thm}\label{thm:main gc}
    There is a morphism of graded Lie algebras 
    \[
    \ogrt_1 \to H(\GC_1)
    \]
    that induces an isomorphism on the parts of degree and loop order $k\geq 0$, i.e., for $k\geq 0$ we have
    \[
    H^{k}(\GC_1^{k\lp}) \cong \ogrt_1^{(k)}.
    \]
    In particular, 
    \[
    H^{k}(\GC_1^{k\lp})
    =
    \begin{cases}
        \Q & \text{$k=3,6$} \\
        0 & \text{otherwise}
    \end{cases}.
    \]
\end{thm}
We also recall the known vanishing results \cite[Lemma~1.4]{Wil19}
\[
H^{j}(\GC_1^{k\lp}) = 0\qquad \text{ for $k\geq 3$ and either $j<k$ or $j>2k-3$}.
\]
Hence Theorem \ref{thm:main gc} computes the odd graph cohomology in the bottom degree, just above the vanishing range.
We also note that our result provides the first infinite family of known cohomology groups of the odd graph complex inside the (potentially) non-vanishing range. 

\subsection*{AI usage statement}
The main concepts of this work predate the arrival of useful AI tools.
However, we have used AI tools in all later stages of this project. Ideas of the proof below are furthermore variations of ideas used in the version \cite{WillwacherOpenAI} of the proof of the even Deligne--Drinfeld conjecture by OpenAI \cite{OpenAI}.

\section{Proof of the odd Deligne-Drinfeld conjecture}

\subsection{ABC coordinates}
\label{sec:coordinates}
When discussing $\ogrt_1$, the cohomological degrees of objects will always equal the weight.
In the following we hence always declare the weights only, with the tacit understanding that the cohomological degree is the same. Since we always work with graded Lie and associative algebras, we will often omit ``graded'' from the notation, and understand that all objects are graded by default.

Recall (see e.g. \cite[Example~8.7.5]{Neisendorfer}) that for graded generators $\{x, x_\alpha \}$ where $x$ is odd, the kernel of the projection $\Lie(x,x_\alpha) \to \Q x$ onto the abelian Lie algebra $\Q x$ is a free Lie algebra on the generators $x_\alpha, x^2, [x,x_\alpha]$.
Let $A,C$ be even generators of weight two and $B$ an odd generator of weight one. We thus have that
\begin{equation}\label{eq:iota}
\begin{gathered}
 \iota:\Lie\langle A,B,C\rangle\longrightarrow\Lie\langle x,y\rangle \\
 \iota(A)=x^2,\qquad \iota(B)=y,\qquad \iota(C)=[x,y].
 \end{gathered}
\end{equation}
defines an isomorphism onto $\mathcal K := \ker \left(\Lie(x,y) \to \Q x \right)$. 
In particular, we have a decomposition of graded vector spaces $\Lie(x,y) = \Q x \oplus \iota(\Lie(A,B,C))$.

Let $\iota^{-1}: \mathcal K \longrightarrow\Lie(A,B,C)$ denote the inverse of $\iota$.
For $\psi\in \ogrt_1$, Example~\ref{ex:lowweight} shows that $\psi \in \mathcal K$ and we can set $P=\iota^{-1}\psi \in \Lie(A,B,C)$, so that 
\[
\psi(x,y)=P(x^2,y,[x,y]).
\]
Then the antisymmetry relation \eqref{eq:oantisymmetry} becomes
\begin{equation}\label{eq:S-abc}
 P(x^2,y,[x,y])+P(y^2,-x,[x,y])=0
 \quad\text{in }\Lie(x,y).
\end{equation}
The pentagon equation is
\begin{align}
0={}&P(t_{12}^2,t_{23}+t_{24},[t_{12},t_{23}]+[t_{12},t_{24}])
 +P(t_{13}^2-[t_{12},t_{23}]+t_{23}^2,t_{34},[t_{13},t_{34}]+[t_{23},t_{34}])\notag\\
 &-P(t_{23}^2,t_{34},[t_{23},t_{34}])\label{eq:P-abc}-P(t_{12}^2+t_{13}^2-[t_{12},t_{23}],t_{24}+t_{34},[t_{12},t_{24}]+[t_{13},t_{34}])\\
 &-P(t_{12}^2,t_{23},[t_{12},t_{23}]).\notag
\end{align}

\subsection{The map \texorpdfstring{$\Phi$}{Phi}}
Define the morphism of graded vector spaces 
\begin{gather*}
    \Phi\colon \ogrt_1 \longrightarrow \Lie(B,C), \qquad
    \psi \longmapsto (\iota^{-1}\psi)(0,B,C).
\end{gather*}

\begin{thm}\label{thm:inj}
    The morphism $\Phi$ is injective and its image is contained in the two-dimensional Lie subalgebra 
    \[
    \Lie([B,C]) \subset \Lie(B,C).
    \]
\end{thm}
We will show this theorem in the following subsections. Theorem~\ref{thm:main} then follows immediately.
\begin{proof}[Proof of Theorem \ref{thm:main}]
  By Example \ref{ex:lowweight} the elements $\sigma_3$ and $\{\sigma_3,\sigma_3\}$ are in $\ogrt_1$.
The map $\Lie(\sigma_3) \to \ogrt_1$ is thus well-defined. It is furthermore injective, since the element $\{\sigma_3, \sigma_3\}$ is non-zero as explained in Example \ref{ex:lowweight}. The dimension of $\ogrt_1$ is at most (and thus exactly) two by Theorem~\ref{thm:inj} so that the map is an isomorphism. This also shows that $\ogrt_1$ is indeed closed under the Ihara bracket.
\end{proof} 

\subsection{A morphism from \texorpdfstring{$\ot_4$}{ot4}}
Define derivations $U,E,\delta,K$ of the free Lie algebra $\Lie( A,B,C)$ by the following assignments on generators
\begin{equation}\label{eq:deri def}
\begin{array}{c|ccc}
 &A&B&C\\\hline
 U&[A,B]&C-2B^2&[A,B]-[B,C]\\
 E&[A,B]&0&[A,B]\\
 \delta&0&-C&[B,C]\\
 K&[A,C]&[A,B]&[A,C]+[A,B^2].
\end{array}
\end{equation}
Here $U,E,\delta$ have degree one and $K$ has degree two. 
Let 
\[
X=\langle U,E,\delta\rangle\subset\Der(\Lie(A,B,C))
\]
be the Lie subalgebra of the Lie algebra of derivations of $\Lie(A,B,C)$ generated by $U,E$ and $\delta$.

\newcommand{\fH}{\mathcal H}
\newcommand{\wfH}{\widetilde{\fH}}

Let $\eta$ be an odd variable of degree $-1$ and let $\Lambda(\eta)=\Q\oplus \Q\eta$ be the free graded commutative algebra generated by $\eta$. Define the graded Lie algebra
\begin{equation}\label{eq:cone-algebra}
 \wfH=( \Lie(A,B,C)\otimes\Lambda(\eta))
       \rtimes(X\oplus\Q \partial_\eta),
\end{equation}
where the Lie bracket on $\Lie(A,B,C)\otimes\Lambda(\eta)$ is extended $\Lambda(\eta)$-linearly and $\partial_\eta$ is the derivative with respect to $\eta$.

\begin{lemma}\label{lem:oneext-representation}
There is a morphism of graded Lie algebras
$\rho:\ot_4\to\wfH$ given on generators by
\begin{equation}\label{eq:oneext-representation}
\begin{aligned}
 \rho(t_{12})&=U+B-C\eta,&
 \rho(t_{13})&=\partial_\eta+(C-B^2-A)\eta,\\
 \rho(t_{23})&=B-C\eta,&
 \rho(t_{24})&=E,\\
 \rho(t_{14})&=-\partial_\eta+\delta+(C-B^2)\eta,&
 \rho(t_{34})&=-A\eta.
\end{aligned}
\end{equation}
As usual, $\rho(t_{ji})=-\rho(t_{ij})$.
\end{lemma}
\begin{proof}
Let 
$$
h=C-B^2-A.
$$
By evaluation on the generators $A,B,C$ one shows
\begin{equation}\label{eq:derivation-identities}
\begin{gathered}
 Uh=Eh=\delta h=0,\qquad U^2=-\ad_h\\
 K=[U,E],\qquad [U,\delta]=[E,\delta]=-K.
\end{gathered}
\end{equation}
Using \eqref{eq:derivation-identities}, the three brackets
$[\rho(t_{12}),\rho(t_{34})]$, $[\rho(t_{13}),\rho(t_{24})]$ and $[\rho(t_{14}),\rho(t_{23})]$ are zero.
The other brackets are given by the following table:
\[
\begin{array}{c|c}
 ijk &[\rho(t_{ij}),\rho(t_{ik})]=[\rho(t_{ik}),\rho(t_{jk})]=-[\rho(t_{ij}),\rho(t_{jk})]\\\hline
 123&-C+([A,B]+[B,C])\eta\\
 124&-K+[A,B]\eta\\
 134&A\\
 234&-[A,B]\eta.
\end{array}
\]
The table verifies all three-index
relations, and hence all defining relations of $\ot_4$.
\end{proof}

\subsection{Injectivity of \texorpdfstring{$\Phi$}{Phi}}
Let us also define the semidirect product
\begin{equation}\label{eq:H}
 \fH=\Lie(A,B,C)\rtimes\Lie(e,k),\qquad |e|=1,\quad |k|=2,
\end{equation}
where the free Lie algebra $\Lie(e,k)$ acts on $\Lie(A,B,C)$ by $e\mapsto E$
and $k\mapsto K$, see \eqref{eq:deri def}.

\begin{prop}\label{prop:reconstruction}
For $\psi\in\ogrt_1$, write $P=\iota^{-1}\psi$ and
$f=P(0,B,C)$. Then in $\fH$ we have the equation
\begin{equation}\label{eq:oneext-reconstruction}
 P(A,B,C) =  f(e+B,k+C)-f(e,k-A),
\end{equation}
where both sides take values in $\Lie(A,B,C)\subset \fH$.
In particular, $\Phi$ is injective.
\end{prop}
\begin{proof}
We may assume that $P$ has homogeneous weight at least three.
Recall $h=C-B^2-A$ and $K=[U,E]$ and define
\[
 \kappa=h+U^2.
\]
The identities \eqref{eq:derivation-identities} show that $\kappa$ is
central in $\wfH$: its bracket with $\Lie(A,B,C)\otimes\Lambda(\eta)$
is zero because $U^2=-\ad_h$, and $U,E,\delta,\partial_\eta$
all commute with $\kappa$.
Applying the Lie algebra morphism $\rho\colon \ot_4\to \wfH$ of Lemma \ref{lem:oneext-representation} to the pentagon equation \eqref{eq:P-abc} we obtain the following equation in $\wfH$:
\begin{equation}\label{eq:full-pentagon-evaluation}
\begin{aligned}
0={}&P\bigl(A+\kappa-[A,B]\eta,\ B+E-C\eta,
             \ K+C-[2A-C,B]\eta\bigr)\\
 &+P\bigl(-A+[A,B]\eta,\ -A\eta,\ -A+[A,B]\eta\bigr)\\
 &-P\bigl(B^2+[C,B]\eta,\ -A\eta,\ [A,B]\eta\bigr)\\
 &-P\bigl(\kappa-B^2-[C,B]\eta,\ E-A\eta,
             \ K-A-[A,B]\eta\bigr)\\
 &-P\bigl(A+\kappa-[A,B]\eta,\ B-C\eta,
             \ C-[A-C,B]\eta\bigr).
\end{aligned}
\end{equation}
For example, the squared argument in the first and fifth terms is
\[
 \rho(t_{12})^2
 =U^2+C-B^2-[A,B]\eta=A+\kappa-[A,B]\eta.
\]
Observe that all arguments lie in the Lie subalgebra
\[
 \wfH_0=(\Lie(A,B,C)\otimes\Lambda(\eta))\rtimes X\subset\wfH,
\]
that is, no argument contains a derivative $\partial_\eta$. 
There is a Lie algebra projection
$\pi:\wfH_0\to \Lie(A,B,C)\rtimes X$ by setting $\eta=0$.
Applying this projection to \eqref{eq:full-pentagon-evaluation} and omitting the central summand $\kappa$ we obtain
\begin{equation}\label{eq:pent eval}
\begin{aligned}
0&= P(A,B+E,K+C)+P(-A,0,-A)-P(B^2,0,0)-P(-B^2,E,K-A)-P(A,B,C)
\\
&=P(A,B+E,K+C)-P(-B^2,E,K-A)-P(A,B,C).
 \end{aligned}
\end{equation}
In the first two terms, the first argument of $P$ commutes with the other two. Indeed,
\[
 [A,B+E]=[A,K+C]=0,\qquad
 [B^2,E]=[B^2,K-A]=0.
\]
Consequently these first arguments may be set to zero and \eqref{eq:pent eval} becomes 
\begin{equation}\label{eq:prerecon}
P(A,B,C) = f(E+B,K+C)-f(E,K-A).
\end{equation}
Now note that the right-hand side of \eqref{eq:oneext-reconstruction} takes values in $\Lie(A,B,C)$ since both terms contribute the same component $f(e,k)$ in the summand $\Lie(e,k)$, but with opposite signs. But the $\Lie(A,B,C)$-component of the right-hand side of \eqref{eq:oneext-reconstruction} agrees with the right-hand side of \eqref{eq:prerecon}, so that \eqref{eq:oneext-reconstruction} follows.
\end{proof}

For later use we shall also record an analogous result obtained from the transposed pentagon equation.

\begin{prop}
For $\psi\in\ogrt_1$, write $P=\iota^{-1}\psi$ and
$f=P(0,B,C)$ and set $h=C-B^2-A$. Then in $\fH$ we have the equation
\begin{equation}\label{eq:oneext-reconstruction2}
 P(h,-B,C-A)-P(h,-B,C) = f(e,k-A)-f(e,k),
\end{equation}    
with both sides taking values in $\Lie(A,B,C)\subset\fH$.
\end{prop}
\begin{proof}
The pentagon equation \eqref{eq:P-abc}, but with strands $2,3$
exchanged, is
\begin{equation}\label{eq:P-abc2}
\begin{aligned}
0={}&
 P(t_{13}^2,t_{34}-t_{23},
   [t_{13},t_{34}]+[t_{12},t_{23}]) +P(t_{12}^2+t_{23}^2-[t_{12},t_{23}],t_{24},
   [t_{12},t_{24}]-[t_{23},t_{24}])\\
&-P(t_{23}^2,t_{24},-[t_{23},t_{24}])-P(t_{12}^2+t_{13}^2-[t_{12},t_{23}],
   t_{24}+t_{34},[t_{12},t_{24}]+[t_{13},t_{34}])\\
&-P(t_{13}^2,-t_{23},[t_{12},t_{23}]).
\end{aligned}
\end{equation}
Applying again the Lie morphism $\rho\colon \ot_4\to \wfH$ gives
\begin{equation}\label{eq:full-transposed-evaluation}
\begin{aligned}
0={}&P\bigl(h,\ -B+(C-A)\eta,\ C-A-[A-C,B]\eta\bigr)+P(\kappa-h,E,K)\\
 &-P\bigl(B^2+[C,B]\eta,\ E,\ [A,B]\eta\bigr)-P\bigl(\kappa-B^2-[C,B]\eta,\ E-A\eta,
             \ K-A-[A,B]\eta\bigr)\\
 &-P\bigl(h,\ -B+C\eta,\ C-[A-C,B]\eta\bigr).
\end{aligned}
\end{equation}
For the second term, one may use
$\rho(t_{12}-t_{23})=U$, so that its first argument is
$U^2=\kappa-h$ and its third argument is $[U,E]=K$.
As in the previous proof, all arguments are in $\wfH_0$.
Setting $\eta=0$ and omitting the central summands $\kappa$ gives
\begin{equation}\label{equ:P2eval2}
\begin{aligned}
0&=P(h,-B,C-A)+P(-h,E,K)-P(B^2,E,0)-P(-B^2,E,K-A)-P(h,-B,C)
\\&=P(h,-B,C-A)+f(E,K) -0 -f(E,K-A) -P(h,-B,C). 
\end{aligned}
\end{equation}
For the second equality we used the commutation relations
\[
 [h,E]=[h,K]=0,\qquad [B^2,E]=[B^2,K-A]=0.
\]
We obtain
\[
 P(h,-B,C-A)-P(h,-B,C)=f(E,K-A)-f(E,K).
\]
Concluding as in the previous proof, \eqref{eq:oneext-reconstruction2} follows.
\end{proof}

\newcommand{\fG}{\mathcal{G}}
\subsection{Two Lie algebra morphisms and the image of \texorpdfstring{$\Phi$}{Phi}}
Let $z$ be a generator of degree one and $c$ of degree two. On $\Lie(z,c)$ define two degree-one derivations $\partial,d$ by
\[
 \partial z=c,\quad \partial c=0,\qquad
 dz=z^2,\quad dc=[z,c].
\]
They satisfy $\partial^2=d^2=[\partial,d]=0$, as is checked on
$z,c$. We use the semidirect product
\begin{equation}\label{eq:G}
 \fG=\Lie(z,c)\rtimes(\Q\partial\oplus\Q d),
\end{equation}
where the second summand is abelian.

\begin{lemma}\label{lem:G-evaluations}
The following assignments on generators define Lie algebra morphisms
$\lambda_0,\lambda_1:\fH\to \fG$:
\begin{equation}\label{eq:G-evaluations}
\begin{array}{c|ccccc}
 &A&B&C&e&k\\\hline
 \lambda_0&z^2&\partial&c&d-\partial&z^2-c\\
 \lambda_1&c+z^2&z&c+2z^2&z-2d&c.
\end{array}
\end{equation}
\end{lemma}
\begin{proof}
The two factors in \eqref{eq:H} are free, so it suffices to verify
the action relations on $A,B,C$. Put $w=[z,c]$. Substitution in
\eqref{eq:deri def} gives
\[
\begin{array}{c|cc}
 &\lambda_i(E(A),E(B),E(C))&\lambda_i(K(A),K(B),K(C))\\\hline
 i=0&(w,0,w)&([z^2,c],w,[z^2,c])\\
 i=1&(-w,0,-w)&([c,z^2],-w,2[c,z^2]).
\end{array}
\]
These are also the brackets of $\lambda_i(e)$ and $\lambda_i(k)$,
respectively, with $(\lambda_i(A),\lambda_i(B),\lambda_i(C))$.
\end{proof}

\begin{prop}\label{prop:oneext-final-image}
For every $\psi\in\ogrt_1$, we have
$\Phi(\psi)\in\Lie([B,C])$.
\end{prop}
\begin{proof}
Take $\psi$ homogeneous of weight $n\geq3$, and define $P=\iota^{-1}\psi$ and $f(B,C)=P(0,B,C)$
as in Proposition~\ref{prop:reconstruction}.
Applying $\lambda_0$ to \eqref{eq:oneext-reconstruction} we obtain 
\[
P(z^2,\partial, c) = f(d,z^2) - f(d-\partial,-c)
= -f(d-\partial,-c)
\]
The first term vanishes because $d$ commutes with $z^2$ and $d^2=0$.
Antisymmetry \eqref{eq:oantisymmetry} furthermore gives
\[
 P(z^2,\partial,c)
 =\psi(z,\partial)=-\psi(-\partial,-z)=
 -P\bigl((-\partial)^2,-z,[-\partial,-z]\bigr)=-P(0,-z, c)=-f(-z,c).
\]
Consequently
\begin{equation}\label{eq:G-first-identity}
 f(-z,c)=f(d-\partial,-c).
\end{equation}

Since the kernel of $\Lie(B,C)\to\Q B$ is free on $B^2,C,[B,C]$ and $f$ lies in that kernel,
there is a unique $F\in\Lie(u,v,w)$ (with $|u|=|v|=2,|w|=3$) so that
\begin{equation}\label{eq:fF}
 f(B,C)= F(B^2,C,[B,C]).
\end{equation}
Analogously, the elements $z^2,c,[z,c]$ freely generate a Lie subalgebra of
$\Lie(z,c)$. Since
\[
 (d-\partial)^2=0,\qquad [d-\partial,-c]=-[z,c],
\]
equation \eqref{eq:G-first-identity} is equivalent to 
\[
F(z^2, c, -[z,c]) = F(0, -c, -[z,c]),
\]
and after replacing the free generators $(z^2,c,[z,c])$ by $(u,v,-w)$ to
\begin{equation}\label{eq:quotient-skew}
  F(u,v,w)= F(0,-v,w).
\end{equation}
In particular, $F$ is independent of $u$.

Next apply $\lambda_1$ to \eqref{eq:oneext-reconstruction2}. Since $\lambda_1(h)=0$ we obtain
\[
 f(-z,z^2)-f(-z,c+2z^2)=f(z-2d,-z^2)-f(z-2d,c).
\]
Here $(z-2d)^2=-z^2$. Thus $f(-z,z^2)$ and $f(z-2d,-z^2)$ take
values in Lie algebras generated by a single odd element, and
vanish in weight $n\geq3$. We conclude that
\begin{equation}\label{eq:G-second-identity}
 f(-z,c+2z^2)=f(z-2d,c).
\end{equation}

Since $[z-2d,c]=-[z,c]$ and $[z,z^2]=0$, this is equivalently written using $F$ from \eqref{eq:fF} as 
\[
F(z^2,c+2z^2, -[z,c])= F(-z^2, c, -[z,c]).
\]
Replacing the free generators $z^2,c,[z,c]$ again by $u,v,-w$ turns this into
\begin{equation}\label{eq:quotient-flipped}
  F(u,v+2u,w)= F(-u,v,w).
\end{equation}
Write $F(u,v,w)= F_0(v,w)$ using
\eqref{eq:quotient-skew}. Setting $v=0$ and replacing $u$ by $u/2$ in
\eqref{eq:quotient-flipped} gives
\[
  F_0(u,w)= F_0(0,w),
\]
so $F_0$ depends only on $w$.
This proves $f\in\Lie([B,C])$.
\end{proof}

\begin{proof}[Proof of Theorem~\ref{thm:inj}]
Proposition~\ref{prop:reconstruction} proves injectivity, and
Proposition~\ref{prop:oneext-final-image} identifies the required
bound on the image.
\end{proof}


\section{Graph cohomology and the proof of Theorem \ref{thm:main gc}}
\label{sec:graph-cohomology}

\subsection{Recollection: The Kontsevich graph complex}

Fix $n\in \Z$. The Kontsevich graph complex $\GC_n$ is a differential graded vector space of $\Q$-linear series of connected graphs with at least trivalent
vertices, and orientation.
More precisely, we have a direct product decomposition of dg vector spaces 
\[
\GC_n = \prod_{g \geq 2}
\GC_n^{g\lp}.
\]
The finite-dimensional subcomplex $\GC_n^{g\lp}$ is defined to be the space of linear combinations of pairs $(\gamma,o)$, with $\gamma$ a connected graph of loop order $g$ with at least trivalent vertices, and $o$ a choice of one of two possible orientations for that graph. The orientation datum is only relevant to fix signs, depends on $n$, and need not be discussed further for the purposes of this paper. One imposes two relations on the pairs $(\gamma,o)$ spanning $\GC_n^{g\lp}$: Isomorphic oriented graphs are identified, and opposite orientations are identified up to a sign. If $\gamma$ has $v$ vertices and $e$ edges, so that $g=e-v+1$, the pair $(\gamma,o)$ is placed in cohomological degree 
\[
|\gamma| = n(v-1) -(n-1)e.
\]
The differential $\delta$ on $\GC_n$ is vertex splitting. There is furthermore a compatible Lie bracket that makes $\GC_n$ into a dg Lie algebra.
See \cite[Section~3]{WillwacherGrt} or \cite{BW} for more details.
We also recall that in the construction the choice of $n$ only matters up to parity -- more precisely, we have that 
\[
H^k(\GC_{n+2j}^{g\lp}) \cong H^{k+2gj}(\GC_n^{g\lp}).
\]

Similarly, we consider the hairy graph complexes $\HGC_{m,n}$ that were introduced by Turchin \cite{Turchin2010}, \cite{AroneTurchin2015}, see also \cite{FTW17} for the conventions used here.
Again $\HGC_{m,n}$ is a dg vector space spanned by graphs with orientations, but now the graphs may have hairs, such as in the following picture.
\[
\begin{tikzpicture}[scale=.75,baseline=-.65ex, rotate=90]
\node[int] (a) at (0,1) {};
\node[int] (b) at (-1,-.6) {};
\node[int] (c) at (1,-.6) {};
\node[int] (d) at (0,-.05) {};
\node (h) at (0,1.8) {};
\draw (a)--(b)--(c)--(a) (a)--(d)--(b) (d)--(c) (a)--(h);
\end{tikzpicture}
\]
The differential is again vertex splitting. It respects the loop and hair number. 
For the purposes of this paper only the $1$-hair subcomplexes of fixed loop number $g$ 
\[
\HGC_{m,n}^{1\text{-hair},g\lp} \subset \HGC_{m,n}
\]
will be relevant. They are independent of the index $m$, which only becomes relevant for graphs with multiple hairs. The cohomological degree of an oriented 1-hairy graph $(\Gamma,o)$ in $\HGC_{m,n}^{1\text{-hair},g\lp}$ of loop order $g=e-v$ with $v$ vertices (excluding the hair vertex) and $e$ edges (including the hair edge) is 
\[
|\Gamma|=n v -(n-1)e.
\]
The hairless graph complex $\GC_n$ and the one-hair subcomplex of $\HGC_{m,n}$ are related by the following result.

\begin{prop}[Turchin--Willwacher {\cite[Theorem~1 and Section~8.2, footnote~10]{TW17}}]
\label{prop:one-hair-comparison}
For each $m$, $n$ there is a morphism of complexes $\GC_n\to \HGC_{m,n}[1]$ given by summing over all ways of adding a hair at a vertex, that induces an injective map on cohomology
\begin{equation}\label{eq:graph-hair-injection}
H^k(\GC_n^{g\lp})\hookrightarrow 
H^{k+1}(\HGC_{m,n}^{1\text{-hair},g\lp})
\end{equation}
for any $k$ and $g\geq2$.
\end{prop}

\subsection{Recollection: One-hair graphs and the Harrison complex}
Let us briefly recall how the hairy graph complexes arise in operad theory.
Let $\Pois_n^c$ be the $n$-Poisson cooperad, which for $n\geq 2$ is the cohomology cooperad of the little $n$-disks operad.
Then it has been shown in \cite[Theorems~11--12 and Proposition~6.1]{FTW17} that, for $n\geq m\geq2$, the derived mapping spaces in the category of Hopf $\Lambda$-cooperads $\Map^h(\Pois_n^c,\Pois_m^c)$ are weakly equivalent to the Maurer--Cartan space of the hairy graph complex $\HGC_{m,n}$. Several precursors of that statement have been shown by Turchin \cite{Turchin2010}, Arone--Turchin \cite{AroneTurchin2015} and the second author \cite{WillwacherGrt}. 
To understand the relation to our paper, recall that the odd Drinfeld--Kohno Lie algebra $\ot_r$ is a graded Lie algebra, with the cohomological grading equal to the weight grading. The totality of the $\ot_r$ for varying $r$ forms an operad in graded Lie algebras $\ot$. This is a member of the family of Drinfeld--Kohno operads $\mathfrak p_n$ of \cite[Theorem~4.6]{FTW17}, more precisely $\ot=\mathfrak p_1$ and $\ot_r=\mathfrak p_1(r)$.\footnote{We note that \cite{FTW17} uses homological conventions while we use cohomological conventions, so the degrees are reversed.}
The Chevalley complex of the graded dual $C(\ot^*)$ comes with a quasi-isomorphism into the Poisson cooperad, $C(\ot^*)\to\Pois_1^c$. This resolution is used there in the computation of the derived operadic mapping spaces.

For this paper we will not need to recall the full operadic formalism, but rather take a minimal shortcut.
Let $N\ot_r$ be the intersection of
the kernels of the strand-forgetting maps $\ot_r\to\ot_{r-1}$.
Thus normalization retains the Lie monomials involving every strand.
Let $\Com^\vee$ denote the Koszul dual cooperad of $\Com$, whose
arity-$r$ part is concentrated in degree $1-r$. Consider the normalized
weight-$k$ Harrison complex
\begin{equation}\label{eq:graph-harrison}
 C_k=\prod_{r\geq2}\operatorname{Hom}_{S_r}
 \bigl(\Com^\vee(r),(N\ot_r)^{(k)}[1]\bigr),
\end{equation}
with the differential $d_H$ induced by the strand-doubling maps of $\ot$, see
\cite[Section~6.2]{WillwacherGrt}. Here ``$[1]$'' lowers cohomological
degree by one. The arity-$r$ term therefore has total degree
\begin{equation}\label{eq:graph-harrison-degree}
 q=k+r-2.
\end{equation}
The Harrison complex above is quasi-isomorphic to the one-hair graph complex: this is the first Hodge summand of
Turchin's comparison \cite[Proposition~9.5, Theorem~10.1 and
Remark~4.8]{Turchin2010}. On the graph side,
Hodge degree counts hairs. On the cosimplicial side, its first summand
is the normalized Harrison complex, up to conventional regrading.
An explicit model for the same comparison is given in
\cite[Theorem~2.2 and Sections~5.1--5.2, especially
equations~(5.6)--(5.7)]{AroneTurchin2015}; its Hodge-degree-one
summand gives \eqref{eq:graph-harrison} after a regrading.
We just recall the following result from the cited works.

\begin{thm}[Turchin {\cite[Proposition~9.5, Theorem~10.1 and Remark~4.8]{Turchin2010}}]\label{thm:HarrisonHGC}
    For any $m$, $k\geq1$ and $q$ there is an isomorphism
    \[
    H^q(C_k) \cong H^q(\HGC_{m,1}^{1\text{-hair},k\lp}).
    \]
\end{thm}

\subsection{The bottom degree and the odd Grothendieck--Teichm\"uller Lie algebra}
\begin{lemma}\label{lem:graph-harrison}
For $k\geq3$, the complex $C_k$ of \eqref{eq:graph-harrison} is concentrated in cohomological degrees $q\geq k+1$. 
The space of weight-$k$ cocycles in $C_k$ of degree $q=k+1$ is isomorphic to $\ogrt_1^{(k)}$.
\end{lemma}
\begin{proof}
Since $\ot_2$ is the free graded Lie algebra on one odd generator,
it has no part of weight $k\geq3$. Hence the arity-two term of $C_k$ vanishes,
and \eqref{eq:graph-harrison-degree} gives $q\geq k+1$, with equality
precisely in arity $r=3$.

For the second statement, let us first write explicitly
\[
C_k^{k+1}
\cong
\Hom_{S_3}(\Com^\vee(3),(N\ot_3)^{(k)})
\cong
(V_2\otimes(N\ot_3)^{(k)})^{S_3},
\]
where $V_2$ is the irreducible two-dimensional $S_3$-representation;
the first isomorphism here forgets the degree shifts.
Put differently, elements of $C_k^{k+1}$ are Lie expressions
$\phi(t_{12},t_{13},t_{23})\in\ot_3$ of weight $k$, involving every
strand, such that
\begin{align*}
0&=\phi(t_{12},t_{13},t_{23})
  +\phi(t_{32},t_{31},t_{21}),\\
0&=\phi(t_{12},t_{13},t_{23})
  +\phi(t_{23},t_{21},t_{31})
  +\phi(t_{31},t_{32},t_{12}).
\end{align*}
The closedness condition $d_H\phi=0$ is
\begin{align*}
0={}&\phi(t_{23},t_{24},t_{34})
-\phi(t_{13}+t_{23},t_{14}+t_{24},t_{34})\\
&+\phi(t_{12}+t_{13},t_{14},t_{24}+t_{34})
-\phi(t_{12},t_{13}+t_{14},t_{23}+t_{24})
+\phi(t_{12},t_{13},t_{23}).
\end{align*}
This is the arity-four Harrison differential, in the same convention
as in \cite[Section~6.2, proof of Lemma~6.4]{WillwacherGrt}.

It remains to identify these expressions with Lie expressions in two
variables. Put $x=t_{12}$, $y=t_{23}$ and $z=t_{13}$. The braid
relations present $\ot_3$ as
\[
\ot_3=\Lie(x,y)\rtimes\Lie(z),
\qquad [z,x]=[z,y]=-[x,y],
\]
where $z$ acts on the free graded Lie algebra $\Lie(x,y)$ by the
unique odd derivation with these values. Since $\Lie(z)$ is
concentrated in weights one and two, every weight-$k$ element of
$\ot_3$, for $k\geq3$, has a unique expression
$\phi=\psi(t_{12},t_{23})$, with $\psi\in\Lie(x,y)_k$.
It is automatically normalized, since every strand-forgetting map
takes values in $\ot_2$, whose weight-$k$ part is zero.
Under this identification the two symmetry relations become the antisymmetry relation
\eqref{eq:oantisymmetry} and the 3-cycle relation \eqref{eq:othree-cycle}, and the
closedness condition becomes the pentagon \eqref{eq:opentagon}. This proves the
assertion.
\end{proof}

\subsection{Proof of Theorem ~\ref{thm:main gc}}
For odd $n$, changing $n=1$ to $n$ shifts the degree of a one-hair graph of loop order $k$ by $-(n-1)k$. Combining Theorem \ref{thm:HarrisonHGC}, Lemma \ref{lem:graph-harrison} and Theorem \ref{thm:main} we obtain the following immediate corollary.
\begin{cor}\label{cor:HGC}
    For any odd $n$, any $m$ and any $k\geq 3$ there are isomorphisms
    \[
    H^{k+1-(n-1)k}
    (\HGC_{m,n}^{1\text{-hair},k\lp})
    \cong \ogrt_1^{(k)}
    \cong 
    \begin{cases}
        \Q &\text{for $k=3,6$} \\
        0 & \text{otherwise}
    \end{cases}.
    \]
\end{cor}

Finally, we can finish the proof of our second main result.
\begin{proof}[Proof of Theorem~\ref{thm:main gc}]
By Proposition~\ref{prop:one-hair-comparison} and Corollary \ref{cor:HGC} we conclude
\begin{equation}\label{eq:graph-upper-bound}
 \dim H^k(\GC_1^{k\lp})\leq\dim\ogrt_1^{(k)}
 \qquad(k\geq3).
\end{equation}
By direct low-loop-order computation \cite[Section~4.2 and Figure~1]{BW}
it is known that equality holds in the only possibly nonzero cases \(k=3,6\). 
Concretely, the loop order $k=3$-class is represented by the tetrahedron graph $K_4$, which in $\GC_1$ is of degree 3.
Brown--Hu--Panzer prove that its self-bracket is
nonzero in $H^6(\GC_1^{6\lp})\cong H^{-6}(\GC_3^{6\lp})$: for the explicit graph homology
cycle $X$ in \cite[Figure~4]{BHP26}, they compute, in their orientation
and pairing conventions,
\begin{equation}\label{eq:tetrahedron-pairing}
 \langle[K_4,K_4],X\rangle=384\neq0
 \qquad\text{\cite[Remark~6.17]{BHP26}}.
\end{equation}
It follows that both $[K_4]$ and its self-bracket are nonzero.

By Theorem~\ref{thm:main}, the assignment $\sigma_3\mapsto[K_4]$
extends uniquely to a morphism of graded Lie algebras
\[
 \ogrt_1\longrightarrow H(\GC_1).
\]
Its image lies on the diagonal on which weight (loop order) and cohomological degree agree. The upper bound
\eqref{eq:graph-upper-bound} and Theorem~\ref{thm:main} bound the
dimension of the diagonal by one in loop orders three and six and
by zero in every other loop order $k\geq3$.
Equation~\eqref{eq:tetrahedron-pairing} shows that these bounds are
attained by the images of $\sigma_3$ and $\{\sigma_3,\sigma_3\}$.

There are no connected graphs with all vertices at least trivalent
in loop orders zero and one. At loop order two, the inequality
$3v\leq2e$ gives $v\leq2$, so degree two in $\GC_1$, which would
require $v=3$, is absent as well. The corresponding weights of
$\ogrt_1$ vanish trivially in weight zero and by
Example~\ref{ex:lowweight} in weights one and two. This proves the
theorem in all loop orders.
\end{proof}

\begin{rem}\label{rem:graph-vanishing}
For $k\geq3$ the known vanishing ranges are
\begin{align*}
 H^j(\GC_1^{k\lp})&=0
 &&\text{if }j<k\text{ or }j>2k-3,\\
 H^q\bigl(\HGC_{m,1}^{1\text{-hair},k\lp}\bigr)&=0
 &&\text{if }q<k+1\text{ or }q>2k-1.
\end{align*}
The lower one-hair bound follows from the Harrison model above,
and \eqref{eq:graph-hair-injection} implies the lower hairless bound.
The hairless range is also recorded in \cite[Lemma~1.4]{Wil19}.
The upper bounds hold on chains: trivalence gives $3v\leq2e$ and
$3v+1\leq2e$, respectively. The restriction $k\geq3$ excludes the
theta class in loop order two, which has degree one in $\GC_1$ and
degree two after adding a hair.
\end{rem}

\bibliographystyle{alpha}
\bibliography{refs}

\end{document}